\documentclass[11pt]{amsart}
\usepackage{amssymb,amsmath}

\makeatletter
\theoremstyle{plain}
 \newtheorem{thm}{Theorem}[section]
\newtheorem{thm*}{Theorem}
 \newtheorem{lem}[thm]{Lemma}
 \newtheorem{prop}[thm]{Proposition}
 
 \numberwithin{equation}{section} 
\numberwithin{figure}{section} 
 \theoremstyle{plain}
 \theoremstyle{definition}
 
 \newtheorem{rem}[thm]{Remark}

\newcommand{\fC}{{{\mathfrak C}}}

\newcommand{\C}{{{\mathbb C}}}

\newcommand{\R}{{{\mathbb R}}}

\newcommand{\bH}{{{\bf H}}}

\newcommand{\bz}{{{\bf z}}}
\newcommand{\bv}{{{\bf v}}}
\newcommand{\bw}{{{\bf w}}}

\newcommand{\cR}{{{\mathcal R}}}
\newcommand{\fp}{{{\mathfrak p}}}
\newcommand{\fq}{{{\mathfrak q}}}

\newcommand{\X}{{{\mathbb X}}}

\newcommand{\fH}{{{\mathfrak H}}}

\newcommand{\cV}{{{\mathcal V}}}

\newcommand{\binfty}{{{\bf \infty}}}

\makeatother

\date{\today
}
\begin{document}

\title[Ptolemaean Inequality]{The Ptolemaean Inequality in the closure of complex hyperbolic plane}

\author[I.d. platis \& Nilg\"un S\"onmez]{Ioannis D. Platis \& Nilg\"un S\"onmez}

\email{jplatis@math.uoc.gr, nceylan@aku.edu.tr }
\address{Department of Mathematics and Applied Mathematics\\ 
University of Crete\\
University Campus\\
GR 700 13 Voutes\\ Heraklion Crete\\Greece}

\bigskip

\address{Department of Mathematics\\
Afyon Kocatepe University\\
Gazlig\"ol Yolu Rekt\"orl\"uk E Blok  \\
TR 03200 Afyon Karahisar\\Turkey}

\begin{abstract}
We prove Ptolemaean Inequality and Ptolemaeus' Theorem in the closure complex hyperbolic plane endowed with the Cygan metric.
\end{abstract}

\maketitle

\section{Introduction}
Let $(S,\rho)$ be a (semi-)metric space. The (semi-)metric $\rho$ is called {\it Ptolemaean}  if any four distinct points $p_1$, $p_2$, $p_3$ and $p_4$ in $S$ satisfy the {\it Ptolemaean Inequality}; that is, for any permutation $(i,j,k,l)$ in the permutation group $S_4$  we have
$$
\rho(p_i,p_k)\cdot \rho(p_j,p_l)\le \rho(p_i,p_j)\cdot \rho(p_k,p_l)+\rho(p_j,p_k)\cdot \rho(p_l,p_i).
$$
A subset $\sigma$ of  $S$ is called a {\it Ptolemaean circle} if for any four distinct points $p_1$, $p_2$, $p_3$ and $p_4$ in $\sigma$ such that $p_1$ and $p_3$ separate $p_2$ and $p_4$ we have
$$
\rho(p_1,p_3)\cdot \rho(p_2,p_4)=\rho(p_1,p_2)\cdot \rho(p_3,p_4)+\rho(p_2,p_3)\cdot \rho(p_4,p_1).
 $$ 
Then we say that $\sigma$ satisfies the {\it Theorem of Ptolemaeus}. The prototype of course is Euclidean plane case, which was proved by the Ancient Greek mathematician Claudius Ptolemaeus (Ptolemy) of Alexandria almost 1800 years ago: Inequality holds for any four points of the Euclidean plane and Ptolemaean circles are Euclidean circles. From the times of antiquity it was realised that even in the simple Euclidean case, Ptolemaean Inequality has an intrinsic importance of its own and various generalisations have been given by a variety of researchers since then. In particular, generalisations to much more abstract spaces have appeard in the last 70 years. Illustratively, see \cite{BFW} and \cite{FS} for the case of ${\rm CAT}(0)$ and ${\rm CAT}(-1)$ cases, respectively, as well as \cite{FLS} for the case of spaces of non-positive curvature, \cite{K} for the case of Hilbert geometries and \cite{S} for normed spaces.  

In this paper we investigate the Ptolemaean Inequality and the Theorem of Ptolemaeus in the metric space $(\overline{\bH^2_\C},\rho)$, where $\overline{\bH^2_\C}$ is the compactified complex hyperbolic plane and $\rho$ is the Cygan metric, see Section \ref{sec:prel} for the definitions. Working in a much more general concept, the first author showed in \cite{P1} that both Ptolemaean Inequality and Theorem of Ptolemaeus hold in the boundary of complex hyperbolic plane $\partial\bH^2_\C$, when the latter is endowed with the Kor\'anyi-Cygan metric $d_\fH$, see Section \ref{sec:chpbd}.  The boundary of complex hyperbolic plane is a sphere and the Cygan metric $\rho$ is the natural extension of $d_\fH$ in the whole sphere.
For the proof we use metric cross-ratios and their properties, see Section \ref{sec:cross}. Using cross-ratios, both Ptolemaean Inequality and Theorem of Ptolemaus can be expressed in quite a neat way.

This paper is organised as follows: In the preliminary Section \ref{sec:prel} we state well known facts about complex hyperbolic plane, its boundary, horospherical coordinates and Cygan metric. In Section \ref{sec:ptol-ineq} we prove Ptolemaean Inequality (Theorem \ref{thm:ptol-ineq}) and finally in Section \ref{sec:ptol-thm} we prove Ptolemaeus' Theorem. Both proofs are carried out via metric cross-ratio in $(\overline{\bH^2_\C},\rho)$ and its invariance properties.

\section{Preliminaries}\label{sec:prel}
The material in this section is well known; for further details we refer to \cite{G}, \cite{P}. In Section \ref{sec:chp} we define complex hyperbolic plane and describe its isometries with respect to the Bergman metric. Section \ref{sec:chpbd} is devoted to the boundary of complex hyperbolic plane; in particular, we define the Heisenberg group and the Heisenberg (Kor\'anyi-Cygan) metric. Horospherical coordinates for complex hyperbolic plane are described in Section \ref{sec:horo} and finally, Cygan metric is in Section \ref{sec:cyg}.
\subsection{Complex hyperbolic plane}\label{sec:chp}
Let $\C^{2,1}$ be the vector space $\C^{3}$ with
the Hermitian form of signature $(2,1)$ given by
$$
\left\langle {\bf {z}},{\bf {w}}\right\rangle 
={\bf w}^*J{\bf z}
=\overline{w}_{3}z_{1}+\overline{w}_{2}z_{2}+\overline{w}_{1}z_{3},
$$
where $J$ is the matrix of the Hermitian form: 
$$
J=\left[\begin{array}{ccc}
0 & 0 & 1\\
0 & 1 & 0\\
1 & 0 & 0\end{array}\right].
$$
We consider the following subspaces of ${\C}^{2,1}$:
\begin{eqnarray*}
V_- & = & \Bigl\{{\bf z}\in\C^{2,1}\ :\ 
\langle{\bf z},\,{\bf z} \rangle<0\Bigr\}, \\
V_0 & = & \Bigl\{{\bf z}\in\C^{2,1}\setminus\{{\bf 0}\}\ :\ 
\langle{\bf z},\,{\bf z} \rangle=0\Bigr\}.
\end{eqnarray*}
Let ${\mathbb P}:\C^{2,1}\setminus\{{\bf 0}\}\longrightarrow \C P^2$ be
the canonical projection onto the two-dimensional complex projective space. Then 
{\sl complex hyperbolic plane} ${\bf H}_{\C}^{2}$
is defined to be ${\mathbb P}V_-$ and its boundary
$\partial{\bf H}^2_{\C}$ is ${\mathbb P}V_0$.
Specifically, $\C^{2,1}\setminus\{{\bf 0}\}$ may be covered with three charts
$H_1,H_2,H_3$ where $H_j$ comprises those points in 
$\C^{2,1}\setminus\{{\bf 0}\}$ for which $z_j\neq 0$. It is clear that
$V_-$ is contained in $H_{3}$. The canonical projection
from $H_{3}$ to $\C^n$ is given by 
${\mathbb P}({\bf z})=(z_1z_{n3}^{-1},\,z_2z_{3}^{-1})=z$. Therefore we can write
${\bf H}^n_{\C}={\mathbb P}(V_-)$ as
$$
{\bf H}^2_{\C} = \left\{ (z_1,\,z_2,\,z_3)\in{\C}^3
\ : \ 2\Re(z_1)+|z_2|^2<0\right\},
$$
which is called the {\it Siegel domain model for} ${\bf H}^2_{\C}$. 
There are distinguished points in $V_0$ which we denote by 
${\bf o}$ and $\binfty$:
$$
{\bf o}=\left[\begin{matrix}  0\\ 0\\  1 \end{matrix}\right], \quad
\binfty=\left[\begin{matrix} 1 \\0 \\ 0 \end{matrix}\right].
$$
Then $V_0\setminus\{{\binfty}\}$ is contained in $H_{3}$ and 
$V_0\setminus\{{\bf o}\}$ (in particular $\infty$) is contained
in $H_1$. Let ${\mathbb P}{\bf o}=o$ and ${\mathbb P}\binfty=\infty$. 
Then we can write $\partial{\bf H}^2_{\C}={\mathbb P}(V_0)$ as
$$
\partial{\bf H}^2_{\C}\setminus\{\infty\} 
=\left\{ (z_1,\,z_2,\,z_3)\in{\C}^n
\ : \ 2\Re(z_1)+|z_2|^2=0\right\}.
$$
In particular $o=(0,0)\in\C^2$.

Conversely, given a point $z$ of 
${\C}^2={\mathbb P}(H_{3})\subset\C P^2$ we may 
lift $z=(z_1,z_2)$ to a point ${\bf z}$ in $H_{3}\subset\C^{2,1}$, 
called the {\sl standard lift} of $z$, by writing ${\bf z}$ in non-homogeneous
coordinates as
$$
{\bf z}=\left[\begin{matrix} z_1 \\ z_2 \\ 1 \end{matrix}\right].
$$

The {\it Riemannian metric} on  ${\bf H}_{\C}^{2}$ is defined by the
distance function $\rho$ given by the formula
$$
\cosh^{2}\left(\frac{\rho(z,w)}{2}\right)
=\frac{\left\langle {\bf {z}},{\bf {w}}\right\rangle 
\left\langle {\bf {w}},{\bf {z}}\right\rangle }
{\left\langle {\bf {z}},{\bf {z}}\right\rangle 
\left\langle {\bf {w}},{\bf {w}}\right\rangle }
=\frac{\bigl|\langle {\bf z},{\bf w}\rangle\bigr|^2}
{|{\bf z}|^2|{\bf w}|^2},
$$
where ${\bf z}$ and ${\bf w}$ in $V_-$ are the standard lifts of $z$ and $w$ 
in ${\bf H}^2_{\C}$ and
$|{\bf z}|=\sqrt{-\langle{\bf z},{\bf z}\rangle}$.
Alternatively,
$$
ds^{2}=-\frac{4}{\left\langle {\bf {z}},{\bf {z}}\right\rangle ^{2}}
\det\left[\begin{array}{cc}
\left\langle {\bf {z}},{\bf {z}}\right\rangle  
& \left\langle d{\bf {z}},{\bf {z}}\right\rangle \\
\left\langle {\bf {z}},d{\bf {z}}\right\rangle  
& \left\langle d{\bf {z}},d{\bf {z}}\right\rangle \end{array}\right].
$$
The sectional curvature of 
$ {\bf H}_{\C}^{2}$ is  pinched between $-1$ and $-1/4$.
Also, $\bH^2_\C$ is a complex manifold, the metric is K\"ahler (in fact, it is the {\it Bergman metric}) and  the holomorphic sectional curvature equals to $-1$.

\medskip

Denote by ${\rm U}(2,1)$ the group
of unitary matrices for  the  Hermitian form 
$\left\langle \cdot,\cdot\right\rangle $. 
Each  matrix $A\in{\rm U}(2,1) $ satisfies
the relation $A^{-1}=JA^{*}J$ where $A^{*}$ is the Hermitian transpose of $A$.
The isometry group  of complex hyperbolic plane is  the {\it projective group} ${\rm PU(2,1)}$. Instead, we may use  ${\rm SU}(2,1)$,  which is a three-fold cover of ${\rm PU}(2,1)$.

Two kinds of subspaces of $\bH^2_\C$ are of particular interest, that is $\C$-lines and mainly $\R-$planes. 

A $\C$-{\it line} is an isometric image of the embedding of $\bH^1_\C=\{z\in\C\;|\;\Re(z)<0\}$ into $\bH^2_\C$. We may assume that the embedding is the standard one
$$
z\mapsto(z,0).
$$
The isometries preserving a $\C$-line is a subgroup isomorphic to ${\rm U}(1,1)$.

An $\R$-{\it plane} (or {\it Lagrangian plane}) $\cR$ is a real 2-dimensional subspace of $\bH^2_\C$ characterised by $\langle \bv,\bw\rangle\in\R$ for all $\bv,\bw\in\cR$. Any real plane $\cR$ is  the isometric image 
of an embedded copy of $\bH_\R^2=\{(x_1,x_2)\in\R^2\;:$ $\;2x_1+x_2^2<0\}$ into $\bH^2_\C$; here, we may assume that the embedding is the standard one:
$$
(x_1,x_2)\mapsto(x_1,x_2,0).
$$
The isometries preserving the $\R$-plane above is a subgroup isomorphic to ${\rm PO}(2,1)$.

\subsection{The boundary: Heisenberg group}\label{sec:chpbd}
A finite point $z$ is in the boundary of the Siegel domain if its standard
lift to $\C^{2,1}$ is 
$$
{\bf z}=\left[\begin{matrix} z_1 \\ z_2 \\ 1 \end{matrix}\right],
\quad \text{ where }\quad 2\Re(z_1) + |z_2|^2 = 0.
$$
We set $$
\zeta=z_{2}/\sqrt{2},\quad \zeta=z_2\in\C.
$$ 
Therefore
$$
{\bf z}
=\left[\begin{matrix} -|\zeta|^2+iv \\ \sqrt{2}\zeta \\1\end{matrix}\right].
$$
In this way, the boundary of the Siegel is identified to the one point compactification of $\C\times\R$, that is the sphere $S^3$.
The action of the stabiliser of infinity ${\rm Stab}(\infty)$ gives this set the structure of a  group which we shall denote by $\fH$. The multiplication for $\fH$ is
$$
(\zeta,v)*(\zeta',v')=\left(\zeta+\zeta',v+v'+2\Im(\overline{\zeta'}\zeta)\right).
$$
We call $\fH$ the {\it Heisenberg group}; $\partial\bH^2_\C=\fH\cup\{\infty\}$.
The {\it Kor\'anyi gauge} $|\cdot|_\fH$ defined on $\fH$ is given by
$$
\left|(\zeta,v)\right|_\fH=\left|-|\zeta|^2+v\right|^{1/2},
$$
where on the right hand side we have the Euclidean norm. Notice that $|\cdot|_\fH$ is not a norm in the usual sense; however, from this gauge we obtain  a metric on $\fH$, the {\it Heisenberg metric  $d_\fH$}, which is defined by the relation
$$
d_\fH\left((\zeta,v),\,((\zeta',v')\right)
=\left|((\zeta,v)^{-1}*((\zeta',v')\right|_\fH.
$$
We remark that the Heisenberg metric is not a path metric.
By taking the standard lift of points on $\partial{\bf H}^2_\C\setminus\{\infty\}$ 
to $\C^{2,1}$ we can write the  metric $d_\fH$ as:
$$
d_\fH\left((\zeta,v),\,(\zeta',v')\right)
=\left|\left\langle\left[ \begin{matrix}
-|\zeta|^2+v \\ \sqrt{2}\zeta\\1 \end{matrix}\right],\,
\left[ \begin{matrix}
-|\zeta'|^2+v' \\  \sqrt{2}\zeta' \\ 1 \end{matrix}\right]
\right\rangle\right|^{1/2}.
$$
The  metric $d_\fH$ is invariant under {\it left translations} and {\it rotations about the axis} $\cV=\{0\}\times\R$. Left translations are essentially the left action of $\fH$ on itself: Given a point $(\zeta',v')\in\fH$ we define
$$
T_{(\zeta',v')}(\zeta,v)=(\zeta',v')*(\zeta,v).
$$
Rotations come from the action of ${\rm U}(1)$ on $\C$: Given a $\theta\in\R$ we define
$$
R_\theta(\zeta,v)=(e^{i\theta}\zeta,v).
$$
These actions form the group ${\rm Isom}^+(\fH,d_\fH)$ of $d_\fH$-{\it orientation-preserving isometries}; this acts transitively on $\fH$. The full group of $d_\fH$-isometries comprises orientation-preserving isometries followed by the conjugation $(\zeta,v)\mapsto(\overline{\zeta},-v)$, Note also that the stabiliser of $0$ consists of rotations. All the above transformations are extended naturally (and uniquely) on the boundary $\partial\bH^2_\C$, by requiring the extended transformations to map $\infty$ to itself.

Now, if $\delta\in\R_*^+$ we define
$$
D_\delta(\zeta,v)=(\delta\zeta,\delta^2v),\quad D_\delta(\infty)=\infty.
$$
We call the map $D_\delta$ a {\it dilation}. It is clear that for every $(\zeta,v),(\zeta',v')\in\partial\bH_\C^2$ we have
$$
d_\fH\left(D_\delta(\zeta,v),D_\delta(\zeta',v')\right)=\delta\;d_\fH\left((\zeta,v),(\zeta',v')\right).
$$
Therefore $d_\fH$ is  scaled up to multiplicative constants by the action of dilations. The {\it orientation-preserving similarity group (resp. similarity group)} ${\rm Sim}^+(\fH,d_\fH)$ (resp. ${\rm Sim}(\fH,d_\fH)$) is the group comprising orientation-preserving Heisenberg isometries (resp. Heisenberg isometries) and dilations.
Finally, {\it inversion} $I$ is given by
$$
I(\zeta,v)=\left(\zeta(-|\zeta|^2+iv)^{-1}\;,\;\-v\left|-|\zeta|^2+iv\right|^{-2}\right),\;\;\text{if}\;(\zeta,v)\neq o,\infty,\;\quad I(o)=\infty,\;I(\infty)=o.
$$
Inversion $I$ is an involution of $\partial\bH^2_\C$. Moreover, for all $p=(\zeta,v),p'=(\zeta',v')\in\fH\setminus\{o\}$ we have
$$
d_\fH(I(p),o)=\frac{1}{d_\fH(p,o)},\quad d_\fH(I(p),I(p'))=\frac{d_\fH(p,p')}{d_\fH(p,o)\cdot d_\fH(o,p')}.
$$
The group generated from orientation-preserving similarities and inversion is isomorphic to the group ${\rm PU}(2,1)$ of holomorphic isometries of $\bH^2_\C$ with respect to the Bergman metric; each holomorphic isometry can be written as a composition of orientation-preserving similarities and inversion. Given two distinct points on the boundary, we can find an element of  ${\rm PU}(2,1)$  mapping those points to $0$ and $\infty$ respectively; in particular, ${\rm PU}(2,1)$ acts doubly transitively on the boundary. 

\subsection{Horospherical coordinates}\label{sec:horo}
For a fixed $u\in\R^+$ consider all those points $z\in\bH^2_\C$ for which the standard lift $\bz$ satisfies
$\langle\bz, \bz\rangle = −2u$. Equivalently
$$
\bz=\left[\begin{matrix} z_1\\z_2\\1
\end{matrix}\right],\quad
\text{where}\;\; 2\Re(z_1)+|z_2 |^2 = −2u.
$$
By writing again $z_2 = 2\zeta$ we have
$z_1 = −|\zeta|^2 − u + iv$. Thus $z$ corresponds to a point $(\zeta, v, u)\in$ $\C\times\R\times\R^+$:
$$
\bz=
\left[\begin{matrix}−|\zeta|^2-u + iv\\
\sqrt{2}\zeta\\
1
\end{matrix}\right].
$$
Let $H_u$ denote the set of points in $\bH^2_\C$ with $\langle\bz, \bz\rangle = −2u$. This set is called the {\it horosphere
of height} $u$. Clearly $H_u$  carries the structure of the Heisenberg group: A point $z$ in the Siegel domain is consequently canonically identified to $(\zeta, v, u)\in \fH\times\R^+$;
we call $(\zeta, v, u)$ the {\it horospherical coordinates of $z$}. The set of the finite boundary points is the horosphere of height zero: $H_0 = \partial\bH^2_\C\setminus\{\infty\}\simeq\fH$ where the isomorphism is $$\fH\ni(\zeta, v) \mapsto (\zeta, v, 0)\in H_0.$$ Explicitly, we have the following transformations of $H_u$:
\begin{enumerate}
\item Left
Heisenberg translation by $(\zeta', v')$ is given by
\begin{equation}\label{eq:Tu}
T^u_{(\zeta', v')}:(\zeta, v, u)\mapsto(\zeta + \zeta', v + v' + 2\Im(\zeta'\overline{\zeta}), u).
\end{equation}
\item Rotation in an angle $\theta$ is given by
\begin{equation}\label{eq:Ru}
R_\theta^u:(\zeta, v, u)\mapsto(e^{i\theta}\zeta, v,u)
\end{equation}
\item Dilation by $\delta>0$ is given by
\begin{equation}\label{eq:Du}
D_\delta^u:(\zeta, v, u)\mapsto(\delta\zeta,\delta^2v,u).
\end{equation}
\item Inversion $I^u$ is given by
\begin{equation}\label{eq:Iu}
I^u:(\zeta, v, u)\mapsto\left(I(\zeta,v),u\right).
\end{equation}
\item Conjugation $J^u$ is given by
\begin{equation}\label{eq:Ju}
J^u:(\zeta, v, u)\mapsto\left(J(\zeta,v),u\right). 
 \end{equation}
\end{enumerate}
The group $G_u$ comprising compositions of transformations (\ref{eq:Tu})-(\ref{eq:Iu}) is thus a group isomorphic to ${\rm PU}(2,1)$. Any two points $p,q\in H_u$ may be mapped to $\infty, (0,0,u)$, respectively, by an element of $G_u$; therefore $G_u$ acts doubly transitively on $H_u$. The group $\overline{G_u}$ comprises elements of $G_u$ followed by conjugation $J^u$ and is a group isomorphic to $\overline{{\rm PU}(2,1)}$, that is, the group comprising elements of ${\rm PU}(2,1)$ followed by $J$.
\begin{rem}
Two horospheres of strictly positive height $u$ and $u'$ may be mapped onto one another via an element of ${\rm PU}(2,1)$: If with no loss of generality we suppose that $u<u'$, then $D_\delta(H_u)=(H_{u'})$, where $\delta=(u'/u)^{1/2}$. Of course, there is no element of ${\rm PU}(2,1)$ mapping a horosphere of positive height $H_u$ to $H_0$, ${\rm PU}(2,1)$ preserves $V_-$ and $V_0$.
\end{rem}
The {\it horoball $U_u$ of height $u$} is the union of all horospheres of height $t > u$; it is an open topological ball of dimension 4. The complex hyperbolic plane is thus a horoball itself, that is the horoball $U_0$ of height 0.

\subsection{The Cygan metric}\label{sec:cyg}
The {\it Cygan metric} on $\overline{\bH^2_\C}\setminus\{\infty\}$ is an extension of the Heisenberg metric $d_\fH$ on $\fH$ to an incomplete metric $\rho$; this is done by defining
$$
\rho\left((\zeta_1 , v_1 , u_1 ), (\zeta_2 , v_2 , u_2 )\right) = \left||\zeta_1-\zeta_2 |^2 + |u_1-u_2 |-iv_1 + iv_2- 2i\Im(\zeta_1{\overline \zeta_ 2})\right|^{1/2}.
$$
We stress here that this agrees with $\langle\bz_1 , \bz_2\rangle$
if and only if one (or both) of $\bz_1$ or $\bz_2$ is a null vector, that is, it corresponds to a point at the boundary. It is absolutely clear that $\rho$ satisfies the following: If $p=(q,u)=(\zeta,v,u)$ and $p'=(q',u')=(\zeta',v',u')$ then
\begin{enumerate}
\item $\rho(p,p')\ge 0$ and $\rho(p,p')=0$ if and only if $p=p'$;
\item $\rho(p,p')=\rho(p',p)$.
\end{enumerate}
Note also that if $p,p'$ lie in the same horosphere $H_u$ if and only if
$$
\rho(p,p')=d_\fH\left((\zeta,v),(\zeta',v')\right).
$$
It follows that when we consider the restriction $\rho_u$ of $\rho$ on a horosphere $H_u$, then the orientation-preserving isometries of $\rho_u$ are compositions of transformations (\ref{eq:Tu}) and (\ref{eq:Ru}), the orientation-preserving similarities are compositions of transformations (\ref{eq:Tu}), (\ref{eq:Ru}) and (\ref{eq:Du}) and the full group of similarities comprising orientation-preserving similarities followed by conjugation (\ref{eq:Ju}).

It remains to show that $\rho$ satisfies the triangle inequality; by showing this we shall also have a proof that the Heisenberg metric $d_\fH$ on $\fH$ is also a metric. 
\begin{prop}
 The function $\rho: \overline{\bH^2_\C}\setminus\{\infty\}\times\overline{\bH^2_\C}\setminus\{\infty\}\to\R^+$ satisfies the triangle inequality.
\end{prop}
\begin{proof}
We remark first that if $p=(\zeta,v,u)$ and $q=(\zeta',v',u')$, then
$$
\rho(p,q)=\rho\left((T_{(-\zeta',-v')}(\zeta,u),v),(0,0,u')\right).
$$
According to the above remark it suffices to consider the points
$$
p_1=(\zeta_1,v_1,u_1),\quad p_3=(0,0,u_3),\quad p_2=(\zeta_2,v_2,u_2) 
$$
and show that
$$
\rho(p_1,p_2)\le\rho(p_1,p_3)+\rho(p_3,p_2).
$$
For this, we indeed have
\begin{eqnarray*}
\rho^2(p_1,p_2)&=&\left||\zeta_1-\zeta_2|^2+|u_1-u_2|-iv_1+iv_2-2i\Im(\zeta_1\overline{\zeta_2})\right|\\
&\le&\left||\zeta_1|^2+|\zeta_2|^2-2\Re(\zeta_1\overline{\zeta_2})+|u_1-u_3|+|u_2-u_3|-iv_1+iv_2-2i\Im(\zeta_1\overline{\zeta_2})\right|\\
&\le&\left||\zeta_1|^2+|u_1-u_3|-iv_1\right|+2|\zeta_1||\zeta_2|+\left||\zeta_2|^2+|u_2-u_3|-iv_2\right|\\
&\le&\left||\zeta_1|^2+|u_1-u_3|-iv_1\right|+2\left||\zeta_1|^2+|u_1-u_3|-iv_1\right|^{1/2}\left||\zeta_2|^2+|u_2-u_3|-iv_2\right|^{1/2}\\
&&+\left||\zeta_2|^2+|u_2-u_3|-iv_2\right|\\
&\le&\left(\rho(p_1,p_3)+\rho(p_2,p_3)\right)^2.
\end{eqnarray*}
\end{proof}
\begin{rem}\label{rem:triangle-eq}
It is quite useful to remark that in the above case triangle inequality holds as an equality if and only if
$$
u_1=u_2=u_3,\quad v_1=v_2=0,\quad \zeta_1,\zeta_2\in\R,\;\zeta_1\cdot\zeta_2\le 0.
$$
\end{rem}
The group of orientation-preserving similarities ${\rm Sim}^+_\rho$ of the Cygan metric is identified to the subgroup of ${\rm PU}(2,1)$ comprising elements which stabilise $\infty$. Explicitly, the orientation-preserving Cygan isometries comprise orientation-preserving Heisenberg isometries. that is, translations
$$
T_{(\zeta',v')}(\zeta,v,u)=\left(T_{(\zeta',v')}(\zeta,v),u\right),
$$
and rotations
$$
R_\theta(\zeta,v,u)=\left(R_\theta(\zeta,v),u\right).
$$
Notice that both translations and rotations preserve horospheres and their restrictions on each horosphere $H_u$ are $T^u$ and $R^\theta_u$, respectively. Now dilations $D_\delta$, $\delta>0$, are given by
$$
D_\delta(\zeta,v,u)=\left(D_\delta(\zeta,v),\delta^2\cdot u\right).
$$
The full group of Cygan similarities comprise orientation-preserving similarities followed by conjugation $J$, where
$$
J(\zeta,v,u)=(\overline{\zeta},-v,u).
$$
Again, the restriction of $J$ on an arbitrary horosphere $H_u$ is $J^u$. We finally discuss inversion $I$; this is given by
$$
I(\zeta,v,u)=\left(\frac{\zeta}{-|\zeta|^2+iv-u},\; -\frac{v}{\left|-|\zeta|^2+iv-u\right|^2},\;\frac{u}{\left|-|\zeta|^2+iv-u\right|^2} \right).
$$
Inversion $I$ is an involution of $\overline{\bH^2_\C}$; notice that $I(\infty)=o$. Moreover, for all $p=(\zeta,v,u),$ $p'=(\zeta',v',u')\in\overline{\bH^2_\C}\setminus\{o\}$ we have
$$
\rho(I(p),o)=\frac{1}{\rho(p,o)},\quad \rho(I(p),I(p'))=\frac{\rho(p,p')}{\rho(p,o)\;\rho(o,p')}.
$$
The above formula can be proved by carrying out straightforward calculations; however, it follows directly by a more general statement, see for instance Prop. 2.7 in \cite{P-S}. Note further that $I=I^0$ and the restriction of $I$ onto any other horosphere $H_u$, $u>0$, is {\it not} equal to $I^u$. Finally, we remark that inversion $I$ fixes the {\it unit Cygan sphere $S(0,1)$ centred at $o$}:
$$
S(0,1)=\{z=(\zeta,v,u)\;:\;\rho(z,o)=1\}.
$$

\section{Ptolemaean Inequality}\label{sec:ptol-ineq}
In this section we prove Ptolemaean Inequality for the compactified complex hyperbolic plane $\overline{\bH^2_\C}$ endowed with the Cygan metric. For the proof, we use the $\rho$-metric cross-ratio defined in Section \ref{sec:cross}. The Ptolemaean Inequality is then stated and proved (Theorem \ref{thm:ptol-ineq}) in Section \ref{sec:ptol1}.
\subsection{The metric cross-ratio} \label{sec:cross}
Let $\overline{\bH^2_\C}$ be the compactified hyperbolic plane; that is
$$
\overline{\bH^2_\C}=\bH^2_\C\cup\fH\cup\{\infty\}.
$$
We extend the Cygan metric $\rho$ in $\overline{\bH^2_\C}\setminus\{\infty\}$ into a metric  in $\overline{\bH^2_\C}$ which we will again denote by $\rho$, by requiring
$$
\rho(p,\infty)=+\infty,\;\;\text{if}\;\;p\neq\infty,\quad \rho(\infty,\infty)=0.
$$
Denote by $\fC(\overline{\bH^2_\C})$ the set of quadruples of pairwise distinct points of $\overline{\bH^2_\C}$, that is
$$
\fC(\overline{\bH^2_\C})=\left(\overline{\bH^2_\C}\right)^4\setminus\{\text{diagonals}\},
$$
and let $\fp=(p_1,p_2,$ $p_3,p_4)\in\fC(\overline{\bH^2_\C})$ be arbitrary. There are six distances in $(0,+\infty]$ involved:
$$
\rho(p_i,p_j),\quad i,j=1,\dots,4,\;\; i\neq j.
$$
We adopt the convention:
$
(+\infty):(+\infty)=1,
$
and to $\fp$ 
we associate the {\it cross-ratio $\X^\rho(\fp)$} defined by
\begin{equation*}\label{eq:rho-cr}
\X^\rho(\fp)=\frac{\rho(p_2,p_4)}{\rho(p_2,p_3)}\cdot\frac{\rho(p_1,p_3)}{\rho(p_1,p_4)}. 
\end{equation*}
From the discussion in Section \ref{sec:cyg} it follows that the cross-ratio $\X^\rho$ remains invariant under the diagonal action of ${\rm PU}(2,1)$ in $\fC(\overline{\bH^2_\C})$. As a corollary we have that if $p_i$ lie on the same horosphere $H_u$, then $\X^\rho$ is invariant under the action of $G_u$.

For every $i,j,k,l=1,\dots,4$, such that $p_i,p_j,p_k,p_l\in \overline{\bH^2_\C}$ are pairwise disjoint, the following symmetry conditions are clearly satisfied:
\begin{enumerate}
 \item[{(S1)}] 
 \begin{eqnarray*}
 &&
 \X^\rho(p_i,p_j,p_k,p_l)=\X^\rho(p_j,p_i,p_l,p_k)=\X^\rho(p_k,p_l,p_i,p_j).
 \end{eqnarray*}
 (notice that all the above are also equal to $\X^\rho(p_l,p_k,p_j,p_i)),$
 \item[{(S2)}]
 \begin{eqnarray*}
 &&
 \X^\rho(p_i,p_j,p_k,p_l)\cdot\X^\rho(p_i,p_j,p_l,p_k)=1,
   \end{eqnarray*}
   \item [{(S3)}]
    \begin{eqnarray*}
 &&
 \X^\rho(p_i,p_j,p_k,p_l)\cdot\X^\rho(p_i,p_l,p_j,p_k)=\X^\rho(p_i,p_k,p_j,p_l).
   \end{eqnarray*}
 \end{enumerate} 
  Let now $\fp=(p_1,p_2,p_3,p_4)\in\fC(\overline{\bH^2_\C})$ and set
 $$
 \X^\rho_1(\fp)=\X^\rho(p_1,p_2,p_3,p_4),\quad \X_2^\rho(\fp)=\X^\rho(p_1,p_3,p_2,p_4).
 $$
 Then due to properties ${\rm (S1)}$, ${\rm (S2)}$ and  ${\rm (S3)}$, the cross-ratios of all possible permutations of points of $\fp$ are functions of $\X_1^\rho(\fp)$ and $\X_2^\rho(\fp)$. 

\subsection{Ptolemaean Inequality}\label{sec:ptol1}
Ptolemaean Inequality for the metric space $(\overline{\bH^2_\C},\rho)$ can be stated as follows:
\begin{thm}\label{thm:ptol-ineq}
Let $\fp=(p_1,p_2,p_3,p_4)\in\fC(\overline{\bH^2_\C})$ and consider the cross-ratios $\X_i(\fp)$, $i=1,2$. Then the following inequalities hold
\begin{equation}\label{eq:ptol-ineq}
\X_1(\fp)+\X_2(\fp)\ge 1\quad\text{and}\quad \left|\X_1(\fp)-\X_2(\fp)\right|\le 1.
\end{equation}
\end{thm}

For the proof we need the following lemmata:

\begin{lem}\label{lem:1}
If  $\fp=(p_1,p_2,p_3,p_4)\in\fC(\bH^2_\C)$,
then there exist points $p'_1,p'_2,p'_3\in\overline{\bH^2_\C}$ such that if $\fp=(p_1,p_2,p_3,p_4)$ and $\fp'=(p'_1,p'_2,p'_3,o)$, then
$$
\X_i^\rho(\fp)=\X_i^\rho(\fp'),\quad i=1,2.
$$
\end{lem}
\begin{proof}
We write
$$
p_i=(\zeta_1,v_i,u_i),\quad i=1,\dots,4.
$$
With no loss of generality we may assume that $u_4=\min\{u_i,\;i=1,\dots,4\}$. Then we set
$$
p'_i=\left(T_{(-\zeta_4,-v_4)}(\zeta_i,v_i),u_i-u_4\right),\quad i=1,\dots,4.
$$
Clearly, $p'_4=o$. Now,
$$
\rho(p'_i,p'_j)=\rho(p_i,p_j),
$$
for all  $i,j=1,\dots 4$, $i\neq j$; hence the lemma is proved.
\end{proof}
\begin{lem}\label{lem:2}
If $\fp=(p_1,p_2,p_3,p_4)\in\fC(\bH^2_\C)$, 
there exists a $\fq=(p,q,r,\infty)\in\fC(\overline{\bH^2_\C})$ such that 
$$
\X_i^\rho(\fp)=\X_i^\rho(\fq),\quad i=1,2.
$$
\end{lem}
\begin{proof}
Given the quadruple $\fp=(p_1,p_2,p_3,p_4)\in\fC(\bH^2_\C)$, 
we track down the quadruple $\fp'=(p_1',$ $p'_2,p'_3,o)$ which we may obtain from Lemma \ref{lem:1}. Applying the inversion $I$ to points of $\fp'$, we have a quadruple $\fq=(p,q,r,\infty)\in\fC(\overline{\bH^2_\C})$.
\end{proof}

\medskip

\noindent{Proof of Theorem \ref{thm:ptol-ineq}}
In the first place we consider quadruples of the form $\fp=(p,q,r,\infty)$ $\in\fC(\overline{\bH^2_\C})$ and we will show that Inequalities (\ref{eq:ptol-ineq}) hold for these quadruples. Notice that we do not assume any specific conditions for the points other than infinity. Now we have
$$
\X_1^\rho(\fp)=\X(p,q,r,\infty)=\frac{\rho(r,p)}{\rho(r,q)},\quad \X_2^\rho(\fp)=\X(p,r,q,\infty)=\frac{\rho(q,p)}{\rho(q,r)},
$$
and the result follows because from the triangle inequality.

Next, we consider an arbitrary $\fp=(p_1,p_2,p_3,p_4)$ $\in\fC(\overline{\bH^2_\C})$. If one or more of the points of $\fp$ lie on the boundary, then by applying a Heisenberg translation and inversion if necessary, we obtain a quadruple of the form $\fp'=(p,q,r,\infty)$ such that
$$
\X_i^\rho(\fp)=\X_i^\rho(\fq),\quad i=1,2.
$$
If none of the points of $\fp$ belong to the boundary, then from Lemma \ref{lem:2} there exists a quadruple the form $\fp'=(p,q,r,\infty)$ such that
$$
\X_i^\rho(\fp)=\X_i^\rho(\fq),\quad i=1,2.
$$
The proof is complete.
\qed

\section{Ptolemaus' Theorem}\label{sec:ptol-thm}
We prove Ptolemaus' Theorem \ref{thm:ptol-thm} in Section \ref{sec:proof-eq}. An introductory discussion of $\R$-circles is in Section \ref{sec:R}.

\subsection{$\R$-circles}\label{sec:R}
An $\R$-{\it circle $\cR$ of height $u$}  is the intersection of a Lagrangian plane with a horosphere $H_u$, $u\ge 0$. We consider two particular $\R$-circles, namely, the {\it standard $\R$-circle of height $0$} (passing through $0$ and $\infty$), 
$$
\cR_\R^0=\left\{(x,0,0)\in\fH\;|\;x\in\R\right\}.
$$
and the {\it standard $\R$-circle of height 1} (passing through $0$ and $\infty$), 
$$
\cR_\R^1=\left\{(x,0,1)\in H_1\;|\;x\in\R\right\}.
$$
Infinite $\R$-circles, that is, $\R$-circles passing through infinity are straight lines; on the other hand, finite $\R$-circles are more complicated curves, see for instance \cite{G} or \cite{P}. An $\R$-circle $R$ is homeomorphic to the unit circle $S^1$; given four distinct points $p_1$, $p_2$, $p_3$ and $p_4$ in $R$, we say that a pair of these points {\it separates} the remaining pair, if the elements of the latter lie in different components of the set comprising of $R$ minus the initial pair, e.g., $p_1,p_3$   separate the points $p_2,p_4$ if $p_2$ and $p_4$ lie in different components of $R\setminus\{p_1,p_3\}$. See also Section 2.3 in \cite{BS}.
\begin{prop}
 Any $\R$-circle $\R$ may be mapped onto the $\cR_\R^0$ or $\cR_\R^1$ by a map $g$ in a manner so that if $\fp=(p_1,p_2,p_3,p_4)\in\fC(\cR)$ and $\fp'=\left(g(p_1),g(p_2),g(p_3),g(p_4)\right)$, then
 $$
 \X^\rho(\fp)=\X(\fp').
 $$
\end{prop}
\begin{proof}
Suppose that $\cR$ is an $\R$-circle of arbitrary height $u$. If it passes through $\infty$, then applying an element of $G_u$ we may map it on the standard $\R$-circle of height $u$. If $u\neq 0$, by applying the dilation $D_{1/u}$ we have a map from $\cR$ to $\cR_\R^1$. Suppose now that our initial $\cR$ does not pass through infinity. By applying a Heisenberg translation, we map it onto an $\R$-circle of height $u$ that passes through $(0,0,u)$. Applying inversion $I^u$ and a Heisenberg translation and a rotation $R^u$ if necessary, we map the latter onto the standard $\R$-circle of height $u$. All the above transformations preserve the metric cross-ratio $\X^\rho$. The proof is complete.    
\end{proof}

\begin{prop}\label{prop:ptol-R}
All $\R$-circles are Ptolemaean circles.
\end{prop}
\begin{proof}
 According to the previous proposition it suffices to show that the standard $\R$-circle of height $u$, where $u=0$ or $1$, is a Ptolemaean circle. Let $p_i$, $i=1,2,3,4$ points in the standard $\R$-circle of height $u$. We suppose first that $p_1$ and $p_3$ separate $p_2$ and $p_4$; we may assume that
$$
p_1=\infty,
\quad p_2=(x_2,0,u),\quad p_3=(x_3,0,u),\quad p_4=(0,0,u),
$$
where $x_2>x_3>0$. Let $\fp=(p_1,p_2,p_3,p_4)$; then
$$
\X_1^\rho(\fp)=\frac{x_2}{x_2-x_3},\quad \X_2^\rho(\fp)=\frac{x_3}{x_2-x_3}.
$$
Hence $\X_1^\rho(\fp)-\X_2^\rho(\fp)=1$. The cases where $p_1$ and $p_2$ separate $p_3$ and $p_4$ and $p_1$ and $p_4$ separate $p_2$ and $p_3$ are proved in an analogous manner.
\end{proof}

\subsection{Proof of Ptolemaeus' Theorem}\label{sec:proof-eq}
We now state and prove Ptolemaeus' Theorem.
\begin{thm}\label{thm:ptol-thm}
A curve $\sigma$ in $(\overline{\bH^2_\C},\rho)$ is a Ptolemaean circle if and only if is an $\R$-circle. Explicitly, let $\fp=(p_1,p_2,p_3,p_4)$ be a quadruple of pairwise distinct points lying on a Ptolemaean circle. Then
\begin{enumerate}
 \item $\X_1^{\rho}(\fp)-\X_2^{\rho}(\fp)=1$ if $p_1$ and $p_3$ separate $p_2$ and $p_4$;
 \item $\X_2^{\rho}(\fp)-\X_1^{\rho}(\fp)=1$ if  $p_1$ and $p_2$ separate $p_3$ and $p_4$;
 \item $\X_1^{\rho}(\fp)+\X_2^{\rho}(\fp)=1$ if  $p_1$ and $p_4$ separate $p_2$ and $p_3$.
\end{enumerate}
\end{thm}
\begin{proof}
According to Proposition \ref{prop:ptol-R} we οnly have to prove that if for a given quadruple $\fp=(p_1,p_2,$ $p_3,p_4)\in\fC(\overline{\bH^2_\C})$ we have equality in one of the Inequalities (\ref{eq:ptol-ineq}), then all points of $\fp$ lie on an $\R$-circle. For this, with no loss of generality we may suppose that the equation in question is
$$
\X^\rho_1(\fp)-\X_2^\rho(\fp)=1,
$$
and we shall distinguish two cases. First, one of the points of $\fp$ lies on the boundary and second, no point of $\fp$ lies on the boundary. In the first case, we may assume after a Heisenberg translation and inversion if necessary, that $\fp$ is the quadruple
$$
p_1=\infty, \quad p_2=(\zeta_2,v_2,u_2),\quad p_3=(\zeta_3,v_3,u_3),\quad p_4=(0,0,u_4).
$$
Then $\X^\rho_1(\fp)-\X_2^\rho(\fp)=1$ reads as
$$
\rho(p_4,p_2)=\rho(p_4,p_3)+\rho(p_3,p_2).
$$
As in Remark \ref{rem:triangle-eq}, one shows that this implies
$$
u_2=u_3=u_4=u,\quad v_2=v_3=0,\quad \zeta_2,\zeta_3\in\R,\;\zeta_2\cdot\zeta_3>0,
$$
and therefore $p_i$ belong in the standard $\R$-circle of height $u$; moreover, $p_1$ and $p_3$ separate $p_2$ and $p_4$.

In the case where no point of $\fp$ belongs to the boundary, we may normalise so that
$$
p_1=(0,0,u_1),\quad p_2=(\zeta_2,v_2,u_2),\quad
p_3=(\zeta_3,v_3,u_3),\quad p_4=(\zeta_4,v_4,u_4).
$$
In the case where $u_i=u$, $i=1,\dots,4$, that is, all $p_i$ belong to the same horosphere of height $u$, we obtain the result by applying inversion $I^u$. We now claim that the possibility that $p_i$ do not lie on the same horosphere cannot exist. Assuming the contrary, with no loss of generality we suppose that $u_1=\min\{u_i,\;i=1,\dots, 4\}$ is the strict minimum of $u_i$; i.e., there exists at least one $u_j$, $j=2,3,4$, $j\neq 1$ with $u_1<u_j$. Consider then the auxiliary points
$$
p'_1=o,\quad p'_2=(\zeta_2,v_2,u_2-u_1),\quad
p'_3=(\zeta_3,v_3,u_3-u_1),\quad p'_4=(\zeta_4,v_4,u_4-u_1).
$$
Applying inversion $I$ we obtain the points
$$
\infty,\quad I(p'_2),\quad I(p'_3),\quad I(p'_4).
$$
If $I(p'_4)=(\zeta,v,u)$, applying the Heisenberg translation $T_{(-\zeta,-v)}$ we have the points
$$
\infty,\quad q_2,\quad q_3,\quad (0,0,u).
$$
Since the cross-ratios have remained invariant, we have that the latter points belong to the standard $\R$-circle of height $u$ and $q_2=(x_2,0,u)$, $q_3=(x_3,0,u)$ and  $x_2\cdot x_3>0$. Moving backwards we have
$$
I(p'_2)=(\zeta+x_2,v,u),\quad I(p'_2)=(\zeta+x_3,v,u),\quad I(p'_4)=(\zeta,v,u).
$$
but this is possible only if $I=I^u$, a contradiction.
This completes the proof.
\end{proof}


\begin{thebibliography}{ZZ99}

\bibitem{BFW}
S.~Buckley \& S.~Falk \& D.~Wraith. 
{\sl Ptolemaic spaces and ${\rm CAT(0)}$.}
Glasg. Math. J. 51.02 (2009): 301--314.

\bibitem{BS} S.~Buyalo \& V. Schroeder;
{\sl M\"obius structures and Ptolemy spaces: boundary
at inﬁnity of complex hyperbolic spaces}.
ArXiv:1012.1699v1 [math.MG].

\bibitem{G} W.~Goldman;
{\sl  Complex hyperbolic geometry}.
Oxford Mathematical Monographs. Oxford Science Publications. The Clarendon Press, Oxford University Press, New York, 1999.

\bibitem{FLS}
T.~Foertsch \& A.~Lytchak \& V.~Schroeder.
{\sl Nonpositive curvature and the Ptolemy inequality.}
Int. Math. Res. Notice 2007 (2007): rnm100.

\bibitem{FS}
T.~Foertsch  \& V.~Schroeder. 
{\sl Hyperbolicity, ${\rm CAT(-1)}$-spaces and the Ptolemy Inequality.}
Math. Ann. 350.2 (2011): 339--356.

\bibitem{K}
D.~Kay. 
{\sl The Ptolemaic inequality in Hilbert geometries.}
Pac. J. of Math. 21.2 (1967): 293--301.

\bibitem{P}
J.R.~Parker;
{\sl Complex hyperbolic Kleinian groups}.
Notes, to be published by Cambridge University Press.

\bibitem{P1} I.D.~Platis;
{\sl Cross--ratios and the Ptolemaean inequality in boundaries of symmetric spaces of rank 1.}
Geom. Ded. 169 (2014): 187--208.


\bibitem{P-S} I.D.~Platis \& V.~Schroeder;
{\sl M\"obius rigidity of invariant metrics in boundaries of symmetric spaces of rank-1}.
Arxiv:1406.6770v3 [math.MG].

\bibitem{S}
I.~Schoenberg. 
{\sl A remark on M.M. Day's characterization of inner-product spaces and a conjecture of L.M. Blumenthal.}
Proc. Amer. Math. Soc. 3.6 (1952): 961-964.

\end{thebibliography}
\end{document}